\documentclass[12pt]{amsart}
\usepackage{times,amsfonts,amsmath,amstext,amsbsy,amssymb,
amsopn,amsthm,upref,eucal}
\usepackage[T1]{fontenc}

\usepackage{tikz}

\newcommand \R {{ \mathbb R}}
\newcommand \C {{ \mathbb C}}
\newcommand \Z {{ \mathbb Z}}

\newtheorem{theorem}{Theorem}[section]
\newtheorem {lemma} {Lemma}[section]
\newtheorem {proposition} {Proposition}[section]

\newtheorem{remark}{Remark}[section]
\newtheorem{claim}{Claim}[section]

\newtheorem{conj}{Conjecture}[section]

\begin{document}
\title[$\mathbb Z^2$ parabolic actions]{Cocycle rigidity and splitting for some discrete parabolic actions}
\author[Danijela Damjanovi\'c and James Tanis]{Danijela Damjanovi\'c$^1$  and James Tanis}
\thanks{ $^1$ Based on research supported by NSF grant DMS-1001884  and NSF grant DMS-1150210 }
\subjclass[2012]{}
\address{Department of Mathematics, \\
      Rice University\\6100 Main st\\Houston, TX 77005}
\email{dani@rice.edu}
\address{Department of Mathematics, \\
      Rice University\\6100 Main st\\Houston, TX 77005}
\email{jtanis@rice.edu}

\maketitle

\begin{abstract}
We prove trivialization of the first cohomology  with coefficients in smooth vector fields, for a class of $\Z^2$ parabolic actions on $(SL(2, \mathbb R)\times SL(2, \mathbb R))/\Gamma$,  where  the lattice $\Gamma$ is irreducible and co-compact. We also obtain a splitting construction involving first and second coboundary operators in the cohomology with coefficients in smooth vector fields. 
\end{abstract}


\section{Introduction}
Classification of cohomological obtructions for the horocycle flow provided a tool for proving cohomological rigidity for the $\mathbb R^2$ parabolic action on  $SL(2, \mathbb R)\times SL(2, \mathbb R)/\Gamma$, \cite{M}.  This is the action on $SL(2, \mathbb R)\times SL(2, \mathbb R)/\Gamma$ induced by left translations by the elements of the two-dimensional subgroup $$
\left(\begin{array}{cc}
 1& t\\
  0& 1  
\end{array}
\right)\times \left(\begin{array}{cc}
 1& s\\
  0& 1  
\end{array}
\right)$$ where $s, t\in \mathbb R$. 
The main ingredient in the proof of cohomological rigidity in \cite{M} is the complete and detailed description of cohomological obstructions for the horocycle flow from \cite{FF}. Moreover, in \cite{DK} the first author  constructed a splitting in cohomology with coefficients in smooth vector fields. 
The splitting  consists essentially in obtaining inverse operators for the first and second coboundary operators over the given action.
The splitting construction in  \cite{DK}  uses in an essential way another part of the result proved in \cite{FF}, namely the fact  that the space of distributional obstructions to solving the cohomological equation over the horocycle flow in each irreducible representation is \emph{finite} dimensional.  An important property of the splitting constructed in  \cite{DK}  is that it is \emph{tame}: inverses constructed for the first and second coboundary operators are tame maps with respect to  the collection of Sobolev norms. A map is tame if the loss of the regularity for the image, for data of \emph{any} given regularity, is \emph{fixed}.  Tame splitting construction lead to a perturbation type result in \cite{DK} for the parabolic $\mathbb R^2$ action defined above.  Namely, for $\Gamma$ an irreducible lattice in $SL(2, \mathbb R)\times SL(2, \mathbb R)$, it is proved in \cite{DK} that the $\mathbb R^2$ parabolic action on  $SL(2, \mathbb R)\times SL(2, \mathbb R)/\Gamma$ is \emph{transversally locally rigid}:  the action sits within a smooth \emph{finite dimensional} family of homogeneous actions such that every sufficiently small transversal perturbation of the family intersects the smooth conjugacy class of the parabolic action. The finite dimensional family of homogeneous actions is precisely defined by the finite dimensional cohomology with coefficients in vector fields, over the parabolic $\mathbb R^2$ action. 

In this paper instead of the $\mathbb R^2$ parabolic action on $SL(2, \mathbb R)\times SL(2, \mathbb R)/\Gamma$, we consider its $\mathbb Z^2$ subactions. As before, $\Gamma$ is assumed throughout to be an irreducible co-compact lattice in $SL(2, \mathbb R)\times SL(2, \mathbb R)$. A representative example from the class of these $\Z^2$ parabolic actions is the time-one $\mathbb Z^2$ parabolic action, namely  the $\mathbb Z^2$ action generated by left translations by the following elements: 
$$
\left(\begin{array}{cc}
 1& 1\\
  0& 1  
\end{array}
\right)\times \left(\begin{array}{cc}
 1& 0\\
  0& 1  
\end{array}
\right),\,\,\,
\left(\begin{array}{cc}
 1& 0\\
  0& 1  
\end{array}
\right)\times \left(\begin{array}{cc}
 1& 1\\
  0& 1  
\end{array}
\right)$$

For the horocycle map (the time-one map of the horocycle flow) on $SL(2, \mathbb R)/\Gamma$, the cohomological obstructions have been completely described by the second author in  \cite{T}. It is proved in \cite{T} that the invariant distributions are a complete set of cohomological obstructions. However, unlike the case of the unipotent flow, the space of invariant distributions for the horocycle map in each irreducible representation has \emph{countable infinite} dimension. For functions for which all the obstructions vanish,  the cohomological problem is solved completely. However, the estimates for the solution of the cohomological equation obtained in \cite{T} are \emph{not tame}. 

In what follows we show vanishing of distributional obstructions for $\mathbb Z^2$ parabolic actions, which implies trivial first cohomology with coefficients in smooth functions. Further, even though the space of invariant distributions in each irreducible representation is infinite dimensional,  we prove existence of a  splitting in cohomology by exploiting the decay in the space invariant distributions obtained in \cite{T}.  The splitting we obtain is \emph{not tame}, but this is only due to the lack of tame estimates for the solution of the first coboundary problem. We comment on possible improvement of non-tame to tame estimates following Conjecture 5.1.  Since the action involves unipotent maps, we extend these cohomological results to cohomology with coefficients in vector fields. At the end we state a conjecture concerning perturbations of $\mathbb Z^2$ parabolic actions. 


\subsection{Description of Actions and the induced operators}

Let $G := SL(2, \R) \times SL(2, \R)$ and $\Gamma \subset G$ be a cocompact, irreducible lattice.  Let $U = \left(\begin{array}{ll} 0 & 1 \\ 0 & 0 \end{array}\right).$  Let $U_1 = U \times 0$ and $U_2 = 0 \times U$ be the two commuting stable unipotent vector fields on $G$.   The left action $\alpha$ by a lattice in the upper unipotent subgroup of $G$ can be viewed as \begin{equation}\label{action}\alpha: \,\,((n, m), x) \to \left(\begin{array}{ll} 1 & n \\ 0 & 1 \end{array}\right) \times \left(\begin{array}{ll} 1 & m \\ 0 & 1 \end{array}\right) x,\end{equation} where $(n, m) \in S \Z \times T  \Z$, with $S, T \in \R^+$ fixed, and $x \in G/\Gamma$.   

The first coboundary operators $L_1$ and $L_2$ are 
$$L_1(f) =f\circ \alpha((S, 0), \cdot)-f$$
and 
$$L_2(f)=f\circ  \alpha((0, T), \cdot) -f$$
for any $f \in C^\infty(G/\Gamma)$. 

\subsection{Results}

In Section \ref{s:transfer}, we prove that the first cohomology over the discrete parabolic action $\alpha$ on $G/\Gamma$ is indeed trivial (i.e. contains only constant cocycles), and we obtain Sobolev estimates of the transfer function with respect to the given data. 

\begin{theorem}\label{transfer}
Let $\Gamma \subset G$ be as above.  For any $r \geq 0$, there is a constant $C_{r, \Gamma} > 0$ such that for any $f, g \in C^\infty(G/ \Gamma)$ with zero average and that satisfy $L_1 g = L_2 f$, there is a solution $P \in C^{\infty}(G/\Gamma)$ such that 
$$
L_1 P = f \text{ and } L_2 P = g,
$$  
and 
$$
\| P \|_{r} \leq C_{r, \Gamma} \| f \|_{3r + 6}.
$$ 
\end{theorem}

In Section \ref{s:split}, we prove the splitting for the first cohomology with coefficients in smooth functions: 

\begin{theorem}\label{split} 
Let $\Gamma \subset G$ be as above.  For any $r \geq 0$, there is a constant $C_{r, \Gamma} > 0$ such that for any $f, g, \phi \in C^{\infty}(G / \Gamma)$ with zero average that satisfy $L_2 f - L_1 g = \phi$, then there exists $P \in C^\infty(G / \Gamma)$ such that 
$$
\|f - L_1 P\|_r \leq C_{r, \Gamma} \| \phi\|_{3r + 10},
$$ 
$$
\|g - L_2 P\|_r \leq C_{r, \Gamma}\| \phi\|_{3r + 10},
$$ 
and
$$
\| P \|_r \leq C_{r, \Gamma}\| f\|_{9 r + 25}.
$$
\end{theorem}

In Section \ref{s:splitVF}, the above results are extended to cohomology with coefficients in smooth vector fields. 

\section{Preliminaries}\label{Prelim}

\subsection{Representation Spaces}

Let $L^2(G / \Gamma)$ be the space of complex-valued square integrable functions with respect to the $SL(2, \R) \times SL(2, \R)$-invariant volume form for $G / \Gamma$.  The left regular representation $L^2(G / \Gamma)$ of $SL(2, \R) \times SL(2, \R)$ decomposes into irreducible, unitary representations as follows: 
\begin{equation}\label{intdec}
L^2(G/\Gamma) = \int_{\oplus_{\mu, \theta}} \mathcal{H}_\mu\otimes \mathcal{H}_\theta.
\end{equation}

The Lie algebra of $sl(2, \R)$ has basis elements 
$$
U = \left(\begin{array}{cc} 0 & 1 \\ 0 & 0 \end{array}\right), 
\ \ V = \left(\begin{array}{cc} 0 & 0 \\ 1 & 0 \end{array}\right), 
\ \ X = \left(\begin{array}{cc} 1/2 & 0 \\ 0 & -1/2 \end{array}\right),
$$
consisting of the generators for the stable and unstable horocycle flows and the geodesic flow.  
The Lie algebra of $G$ is $sl(2, \mathbb{R})\times sl(2, \mathbb{R})$, which has the basis 
$$
\begin{array}{ccc}
U_1 = U \times I, \ \ U_2 = I \times U, \\
V_1 = V \times I, \ \ V_2 = I \times V, \\
X_1 = X \times I, \ \ X_2 = I \times X.
\end{array}
$$
The \textit{Laplacian} determined by this basis is the elliptic element of the enveloping algebra of $sl(2, \R) \times sl(2, \R)$ defined as 
$$
\triangle = -(U_1^2 + X_1^2 + V_1^2 + U_2^2 + X_2^2 + V_2^2).
$$ 
The elements of $sl(2, \R)$ are skew-adjoint, and thus $\triangle$ is an essentially self-adjoint operator.  
 Let $\triangle_1 = - (U^2 + X^2 + V^2).$  It will sometimes be useful to use the equivalence of operators 
 $$
 (I + \triangle)^r \approx (1 + \triangle_1, I)^r + (I, I + \triangle_1)^r,
 $$
 for any $r \geq 0$.
 
 The$\textit{ Sobolev space of}\textit{ order }$$\textit{r }$ $> 0$ of a unitary representation space $\mathcal H$ of $G$ is the Hilbert space $W^r(\mathcal{H})\subset \mathcal{H}$ defined to be the maximal domain determined by the inner product 
 $$
 \langle f, g\rangle_{W^r(\mathcal{H})} := \langle (1 + \triangle)^r f, g\rangle_{\mathcal{H}}.
 $$  

We denote the Casimir operators in the first and second components of $SL(2, \R)\times SL(2, \R)$ by $\Box_1$ and $\Box_2$, respectively.  Because $\Gamma$ is irreducible, each Casimir operator has a spectral gap in the sense that $spec(\Box_j) \cap \R^+$ has a lower bound $\mu_0 > 0$, for $j = 1, 2$ (see Theorem 15 of \cite{M}).  

All operators in the enveloping algebra are $\textit{decomposable}$ with respect to the direct integral decomposition ($\ref{intdec}$), so that for $r \geq 0$, 
$$
W^r(G/\Gamma) = \int_{\oplus_{(\mu, \theta) \in spec(\Box_1) \times spec(\Box_2)}} W^r(\mathcal{H}_\mu\otimes \mathcal{H}_\theta) d\gamma(\mu, \theta)
$$
with respect to the measure $d\gamma(\mu, \theta)$.  
The distributional dual to $W^r(\mathcal{H})$ is $W^{-r}(\mathcal{H})$.  

 The space $C^{\infty}(\mathcal{H}) := \cap _{r \geq 0} W^r(\mathcal{H})$ consists of infinitely differentiable functions in $\mathcal{H}$.   Now the left quasi-regular representation spaces $\{C^{\infty}(G / \Gamma)\} \cup \{W^r(G / \Gamma)\}_{r \geq 0}$ are defined via unitary equivalence.  Let $\left(C^\infty(\mathcal H)\right)'$, $\left(\left(C^\infty(G/\Gamma)\right)'\right)$, be the space of distributions on $C^\infty(\mathcal H)$, (resp. $C^\infty(G/\Gamma)$).  
 
 \subsection{Invariant Distributions}
 
We now define the invariant distributions for the map $e^U$ in a representation space $\mathcal H$ of $SL(2, \R)$.   The space of smooth functions in $\mathcal H$ is defined analogously, as in \cite{FF}.  The $\textit{ Sobolev space of}\textit{ order }$ $\textit{r }$ $> 0$ is the Hilbert space $W^r(\mathcal{H})\subset \mathcal{H}$ defined to be the maximal domain determined by the inner product $$\langle f, g\rangle_{W^r(\mathcal{H})} := \langle (1 + \triangle)^r f, g\rangle_{\mathcal{H}} $$ for $f, g \in \mathcal{H}.$   Set $C^\infty(\mathcal H) := \cap_{r \geq 0} W^r(\mathcal H)$, and let $\left(C^\infty(\mathcal H)\right)'$ be the space of  distributions on $C^\infty(\mathcal H).$

The space of invariant distributions for the map $e^U$ in the representation space $\mathcal H$ of $SL(2, \R)$ is defined by 
$$
\mathcal{I}(\mathcal{H}) := \{\mathcal D \in \left(C^{\infty}(\mathcal{H})\right)' : \mathcal D (f \circ e^U)  = \mathcal D (f), \text{ for all } f\in C^\infty(\mathcal H)\}.
$$

The invariant distributions for the time-one map are all finite regularity.  See Theorem 1.1 of \cite{T} for a more precise description than what we need here.  





\subsection{Induced action on unitary representation space}
The action $(\ref{action})$ is by the lattice $S \Z \times T \Z$ in the upper unipotent subgroup.  In what follows, we will restrict our attention to the cannonical $\Z^2$ action when $T=S=1$, as the results for the $S \Z \times T \Z$ action follow completely analogously. 

 Set 
 \begin{equation}\label{translation}
 u_1 := \left(\begin{array}{cc}1 & 1\\0 & 1\end{array}\right) \times \left(\begin{array}{cc}1 & 0\\0 & 1\end{array}\right) \text{ and } u_2 := \left(\begin{array}{cc}1 & 0\\0 & 1\end{array}\right) \times \left(\begin{array}{cc}1 & 1\\ 0 & 1\end{array}\right).
 \end{equation}
    As before, consider operators $L_1$ and $L_2$ on $L^2(G/\Gamma)$ \begin{equation}\label{L_j} L_1 f(x) := f (u_1 x) - f(x) \text{ and } L_2 f(x) := f( u_2 x)  - f(x). \end{equation}   

Iterating the annihilation and creation operators from $sl(2, \R)$ gives a standard orthogonal basis $\{u_{\mu, j}\}_{j}$ for each irreducible representation $\mathcal H_\mu$ of $SL(2, \R)$.  Taking tensor products, we use the orthogonal basis $\{u_{\mu, j} \otimes u_{\theta, k}\}_{j, k}$ for the irreducible representation space $\mathcal H_\mu \otimes \mathcal H_\theta$ of $SL(2, \R) \times SL(2, \R)$.  The operators $L_1, L_2$ are defined on $\mathcal{H}_\mu\otimes \mathcal{H}_\theta$ via a unitary equivalence $Q_{\mu, \theta} :L^2(G / \Gamma) \to \mathcal H_\mu \otimes \mathcal H_\theta$ according the formula \begin{equation}\label{57}\left\{\begin{array}{ll}Q_{\mu, \theta} L_1 Q_{\mu, \theta}^{-1} f = 
\int_{\oplus_\mu } c_{j, k}^{\mu, \theta} (\pi(u_1) u_{\mu, j}  - u_{\mu, j}) \otimes u_{\theta, k} \\ Q_{\mu, \theta} L_2 Q_{\mu, \theta}^{-1}  f = \int_{\oplus_\mu } c_{j, k}^{\mu, \theta} u_{\mu, j} \otimes (\pi(u_1) u_{\theta, k}  - u_{\theta, k}).\end{array}\right.\end{equation} 

\subsection{Cohomology over $\Z^k$ actions}\label{s:cohomology}
To ease notation for this subsection, we write $M := G/\Gamma.$ Let 
\begin{equation}\nonumber
 \rho:\Z^k\rightarrow Diff ^\infty(M)
\end{equation}
be a smooth $\Z^k$ action on a smooth manifold $M$, preserving a smooth volume form.  A $C^\infty$-diffeomorphism $f:M \rightarrow M$ induces a map on smooth vector fields $X\in \text{Vect}^\infty(M)$ by $f_*X=(Df)\circ X\circ f^{-1}$.  Let $C^l(\Z^k,Vect^\infty(M))$ denote the space of  multilinear maps from $\Z^k\times \dots \times\Z^k$ to $Vect^\infty(M)$.

Then we have the cohomology sequence:
\begin{equation}~\label{cohomology sequence}
 C^0(\Z^k,Vect^\infty(M)) \xrightarrow{\delta_v^1} C^1(\Z^k,Vect^\infty(M))\xrightarrow{\delta_v^2}C^2(\Z^k,Vect^\infty(M)),
\end{equation}
where  the operators $\delta_v^1$ and $\delta_v^2$ are defined as follows. 

For $H \in C^0(\Z^k,Vect^\infty(M))=Vect^\infty(M)$ and $\beta \in C^1(\Z^k,Vect^\infty(M))$ define
\begin{equation}\label{cohom_vector}
\begin{array}{cc}
\delta_v^1H(g) = \rho(g)_*H-H \\ 
(\delta_v^2\beta)(g_1,g_2) = (\rho(g_2)_*\beta(g_1)-\beta(g_1)) - (\rho(g_1)_*\beta(g_2)-\beta(g_2)).
\end{array}
\end{equation}
The first cohomology $H_\rho^1(\mathbb Z^k, Vect^\infty(M))$ over the action $\rho$ with coefficients in $Vect^\infty(M)$ is defined to be $Ker \delta_v^2/Im \delta_v^1$.

If the manifold $M$ is homogeneous and the action $\rho$ is affine, then for $g\in \mathbb Z^k$, the map $\rho(g)_*$ is  linear.  Since $M$ is parallelizable, if there is a basis of vector fields in which all $\rho(g)_*$ are upper triangular with ones on the diagonal, then computing the cohomology $H_\rho^1(\mathbb Z^k, Vect^\infty(M))$ can be inductively reduced to computing the first cohomology with coefficients in smooth functions $H_\rho^1(\mathbb Z^k, C^\infty(M))$ (see Section 4 of \cite{DK}).

The cohomology $H_\rho^1(\mathbb Z^k, C^\infty(M))$ is defined via the sequence
\begin{equation}\label{62}
 C^0(\Z^k,C^\infty(M)) \xrightarrow{\delta^1} C^1(\Z^k,C^\infty(M))\xrightarrow{\delta^2}C^2(\Z^k,C^\infty(M)),
\end{equation}
where the operators $\delta^1$ and $\delta^2$ are defined by:
\begin{equation}\label{cohom_function}
\begin{array}{cc}
\delta^1 f(g) =f\circ \rho(g)-f \\ 
\delta^2\beta(g_1,g_2) = (\beta(g_1)\circ \rho(g_2)-\beta(g_1)) - (\beta(g_2)\circ \rho(g_1)-\beta(g_2)),
\end{array}
\end{equation}
where $f\in C^0(\Z^k, C^\infty(M)) = C^\infty(M)$ and $\beta\in C^1(\Z^k,C^\infty(M))$. Then we have $H_\rho^1(\mathbb Z^k, C^\infty(M))$ over the action $\rho$ with coefficients in $C^\infty(M)$ is defined to be $Ker \delta^2/Im \delta^1$.

We say the first cohomology over $\rho$ with coefficients in $C^\infty(M)$ trivializes if it is isomorphic to $\mathbb R$, i.e. if every $f \in C^\infty(M)$ with zero average is in the image of $\delta^1$. 

We say the first cohomology over $\rho$ with coefficients in $Vect^\infty(M)$ trivializes if it is the same as the cohomology over $\rho$ with coefficients in constant vector fields. 

\section{Proof of Theorem \ref{transfer}}\label{s:transfer}
The argument that follows does not change with the scaling of $S, T$ for the lattice action $S \Z \times T \Z$ by upper unipotent elements described in formula (\ref{action}).  We therefore take $S, T = 1$ for simplicity.

We cite the following well-known Howe-Moore Ergodicity theorem, \cite{FK}.
\begin{theorem}[Howe-Moore Ergodicity theorem]
Suppose that $G$ is a semi-simple Lie group with finite center and has no compact simple factors, and suppose $X$ is an irreducible $G$-space with finite $G$-invariant measure.  If $H$ is a closed, noncompact subgroup of $G$, then $H$ also acts ergodically on $X$.
\end{theorem}

Recall the definition of the operators $u_1$ and $u_2$ in (\ref{translation}).  Note $\overline{\{u_1^k\}}_{k \in \Z}$ and $\overline{\{u_2^k\}}_{k \in \Z}$ are closed, noncompact subgroups, so they act ergodically on $G/\Gamma$.  Hence, the nonconstant components $C \otimes \mathcal H_\theta$ and $\mathcal H_\mu \otimes C$ ($C\in \C)$ do not exist in the decomposition of $L^2(G/\Gamma)$.  In what follows, we only consider $\mathcal H_\mu \otimes \mathcal H_\theta$, where $\mathcal H_\mu$ and $\mathcal H_\theta$ are both nonconstant.
 


Theorem \ref{transfer} concerns the existence of a smooth solution $P$ to the cohomological equation $L_1 P = f$ and $L_2 P = g$  and Sobolev estimates of $P$ in terms of $f$.  To prove it, we consider each irreducible component of $ L^2(G/\Gamma)$ individually.

Theorem  1.1 of \cite{T} implies

\begin{lemma}\label{inv_regularity}
Let $\mathcal H$ be an irreducible unitary representation of $SL(2, \R)$.  Then there is some $0 < r < \infty$ such that $\mathcal {I}(\mathcal H) \subset W^{-r}(\mathcal H)$.
\end{lemma}

Our approach follows \cite{M} and \cite{R}.  As in \cite{R}, for each $\mathcal{D} \in \mathcal{I}(\mathcal{H}_\mu)$ define $\tilde{\mathcal{D}}^1 : W^s(\mathcal{H}_\mu\otimes \mathcal{H}_\theta)\rightarrow\mathbb{C}\otimes \mathcal{H}_\theta$ by 
$$
\tilde{\mathcal{D}}^1 := \mathcal{D}\otimes I,
$$ 
and similarly for each $\mathcal{D} \in\mathcal{I}(\mathcal{H}_\theta),$ define $\tilde{\mathcal{D}}^2$ by \begin{equation}\label{tilde dist}\tilde{\mathcal{D}}^2 := I\otimes \mathcal{D}.\end{equation}  
Now extend $\tilde{\mathcal D}^j$ linearly to finite sums of simple tensors, for $j = 1, 2.$  In formula $(\ref{D_j inf})$ below, $\tilde{\mathcal D}^j$ is extended to be defined on infinite sums of simple tensors.

For each $\mu \in spec(\Box)$, there is an orthogonal basis $\{u_{\mu, j}\}_{j \in A_{\mu}} \in C^\infty(\mathcal{H}_\mu)$ of $\mathcal{H}_\mu$, where 
$$
\mathcal{A}_\mu = \left\{\begin{array}{ll} \Z  \ \ \ \ \ \ \ \ \ \ \ \ \ \ \text{ for } \mu > 0 \\ \Z_{ \geq n} \text{ for } \mu = -n^2 + n \leq 0 \text{ and } n \in \Z_{ > 0}\end{array}\right.
$$ 
(see Section 2.4 of \cite{FF}).

For any 
\begin{equation}\label{f decomp}
 f = \sum_{(j, k) \in A_{\mu}\times A_{\theta}} c_{j, k} u_{\mu, j} \otimes u_{\theta, k} \in C^\infty(\mathcal{H}_\mu \otimes \mathcal{H}_\theta),\end{equation}
with $\{c_{j, k}\}_{(j, k) \in A_{\mu}\times A_{\theta}} \subset \C$ and $\mathcal D \in \mathcal I(\mathcal H_\mu)$,
define 
\begin{equation}\label{D_j inf}
\tilde{\mathcal D}^1 (f) := \sum_{k \in A_\theta} \sum_{j \in A_\mu} c_{j, k} \mathcal{D}(u_{\mu, j}) u_{\theta, k}.\end{equation}  
Define $\tilde{\mathcal D}^2$ analogously.

The operators $\tilde{\mathcal D}^j$ map $C^\infty$ functions to $C^\infty$ functions.

\begin{lemma}\label{D^j regular}
Let $(\mu, \theta) \in spec(\Box_1) \times spec(\Box_2)$, and $f\in C^\infty(\mathcal H_\mu \otimes \mathcal H_\theta)$.  Then for all $\mathcal D \in \mathcal{I}(\mathcal{H}_\mu)$ (resp. $\mathcal{I}(\mathcal{H}_\theta))$, we have $\tilde{\mathcal D}^{j}(f) \in C^\infty(\mathcal H_\theta)$ (resp. $C^\infty(\mathcal H_\mu)$).
\end{lemma}
\begin{proof}

We will carry the argument for $\tilde{\mathcal D}^1$, and the case for $\tilde{\mathcal D}^2$ will follow in the same way.  
Using the finite regularity condition from Lemma \ref{inv_regularity} and the fact that $\{u_{\theta, k}\}_k$ is an orthogonal basis, we have that for any $r \geq 0$, 
$$
\|\sum_{k \in A_\theta} \sum_{j \in A_\mu} c_{j, k} \mathcal{D}(u_{\mu, j}) u_{\theta, k}\|_{W^{r}(\mathcal{H}_\theta)}
$$ 
\begin{equation}\label{equa:sums_prod}
\leq C_{r, \Gamma} \left(\sum_{(j, k) \in A_\mu\times A_\theta} |c_{j, k}|^2 \|u_{\mu, j}\|_{W^{r + 1}(\mathcal{H}_\mu)}^2 \|u_{\theta, k}\|_{W^{r + 1}(\mathcal{H}_\theta)}^2\right)^{1/2}.
\end{equation}
Now Formula (25) of \cite{FF} shows that 
$$
\langle [I - (X^2 + \frac{V^2 + U^2}{2})]^{r + 1} u_{\mu, j}, u_{\mu, j}\rangle_{\mathcal H_\mu} = (1 + \mu + 2 k^2)^{r + 1} \| u_{\mu, j}\|_0,
$$
so that 

\begin{align*}
u &= \arctan x & dv &= 1 \, dx
\\ du &= \frac{1}{1 + x^2} dx & v &= x.
\end{align*}

\begin{align*}
(\ref{equa:sums_prod})& \leq C_{r, \Gamma} \left(\sum_{(j, k) \in A_\mu\times A_\theta} |c_{j, k}|^2 (1 + \mu + 2 j^2)^{r + 1} (1 + \theta + 2 k^2)^{r + 1} \|u_{\mu, j}\|_{\mathcal{H}_\mu}^2 \|u_{\theta, k}\|_{\mathcal{H}_\theta}^2\right)^{1/2}
\\
& \leq C_{r, \Gamma} \left(\sum_{(j, k) \in A_\mu\times A_\theta} |c_{j, k}|^2 [(1 + \mu + 2 j^2)^{2(r + 1)}  + (1 + \theta + 2 k^2)^{2(r + 1)}] \|u_{\mu, j}\|_{\mathcal H_\mu}^2 \|u_{\theta, k}\|_{\mathcal{H}_\theta}^2\right)^{1/2}
\\
& \leq C_{r, \Gamma} \left( \sum_{(j, k) \{in A_\mu \times A_\theta} |c_{j, k}|^2 \|u_{\mu, j} \otimes u_{\theta, k}\|_{2(r + 1)}^2 \right)^{1/2} = C_{r, \Gamma} \| f \|_{2(r + 1)} < \infty. 
\end{align*}

\end{proof}  

\begin{lemma}\label{2}
Let $f, g \in C^{\infty}(\mathcal{H}_\mu \otimes \mathcal{H}_\theta)$ be nonzero functions of zero average, and suppose $f, g$ satisfy $L_2 f = L_1 g$.  Then for all $\mathcal{D} \in \mathcal{I}(\mathcal{H}_\mu)$ (resp. $\mathcal{D} \in \mathcal{I}(\mathcal{H}_\theta)$), we have $\tilde{\mathcal{D}}^1(f) = 0$ (resp. $\tilde{\mathcal{D}}^2(g) = 0$).
\end{lemma}
\begin{proof}

 Let $f$ be as in $(\ref{f decomp})$. 
 From the definitions of $(\ref{L_j})$ and $(\ref{tilde dist})$, we have the relations 
 $$
 \Bigg\{\begin{array}{ccc}
 \tilde{\mathcal D}^2 L_1 = L_1 \tilde{\mathcal D}^2,\\ \ \tilde{\mathcal D}^1 L_2 = L_2 \tilde{\mathcal D}^1,\\ \ \tilde{\mathcal D}^1 L_1 = 0 = \tilde{\mathcal D}^2 L_2\end{array}.
 $$   
 Then the assumption $L_2 f = L_1 g$ implies $L_2 (\tilde{\mathcal D}^1(f)) = 0.$   
  
  Then \begin{equation}\label{58}\tilde{\mathcal{D}}^1(f) = C \ a.e.\end{equation} (constant) by ergodicity.  Because $\mathcal H_\theta$ is not the trivial component and $\tilde{\mathcal D}^1(f) \in C^\infty(\mathcal H_\theta)$, by Lemma $\ref{D^j regular}$, it follows that $\tilde{\mathcal{D}}^1(f) = 0$.  
  
The same argument proves $\tilde{\mathcal{D}}^2(g) = 0$ for all $\mathcal{D}^2 \in \mathcal{I}(\mathcal{H}_\theta)$.  \end{proof}

Define 
$$
Ann(\mathcal I(\mathcal H)) := \{ f\in C^\infty(\mathcal H) : D(f) = 0 \text{ for all } D\in \mathcal I(\mathcal H)\}.
$$  
We will use the following result from \cite{T} in the proof of Theorem \ref{transfer}.

\begin{lemma}\label{3}(\cite{T})
Let $\mathcal{H}$ be any unitary representation of $SL(2, \mathbb{R})$, and suppose the spectral gap condition holds.  Let  $r \geq 0$.  Then there is a constant $C_{r, \mathcal H} > 0$ such that for any $f\in Ann(\mathcal{I}(\mathcal{H}))$ there is a unique solution $P$ in $\mathcal{H}$ to 
$$
L_1 P = f,
$$ 
and 
\begin{equation}\label{301}
\| P \|_{W^r(\mathcal{H})} \leq C_{r, \mathcal{H}}
\| f \|_{W^{3r + 4}(\mathcal{H})}. 
\end{equation}
\end{lemma}

\begin{remark}\label{rmk:Tvariable}
By Theorem 1.2 in \cite{T}, Lemma \ref{3} holds for the time-$T$ map with a constant $C_{r, T, \mathcal H} > 0$. 

Moreover, the proof of Theorem 1.2 in \cite{T} shows that $C_{r, T, \mathcal H}$ is uniformly bounded in $T$ in a neighborhood of $T =  1$.  This fact will be used in Proposition \ref{splitVF}.
\end{remark}

\noindent{\bf Proof of Theorem \ref{transfer}}:

 
 Write $$f = \int_{\oplus (\mu, \theta) \in spec(\Box_1)\times spec(\Box_2)} f_{\mu, \theta} \ d\gamma(\mu, \theta)$$  \text{ and } $$g = \int_{\oplus (\mu, \theta) \in spec(\Box_1)\times spec(\Box_2)} g_{\mu, \theta} \ d\gamma(\mu, \theta),$$ where $f, g \in C^\infty(G / \Gamma).$  Write $f_{\mu, \theta}  = \sum_{(j, k) \in A_\mu\times A_\theta} c_{j, k}^{\mu, \theta} u_{\mu, j} \otimes u_{\theta, k}$ and $g_{\mu, \theta} = \sum_{(j, k) \in A_\mu\times A_\theta} d_{j, k}^{\mu, \theta} u_{\mu, j} \otimes u_{\theta, k},$ where $\{c_{j, k}^{\mu, \theta}\}_{(j, k) \in A_\mu\times A_\theta}, \{d_{j, k}^{\mu, \theta}\}_{(j, k) \in A_\mu\times A_\theta}$ $ \subset \C$ decay rapidly at infinity.

We use the representation parameters $\nu^\mu := \sqrt{ 1 - 4\Re\mu}$ and $\nu^\theta := \sqrt{1 - 4 \Re \theta}$, and we now show that $\sum_{j \in A_{\mu}} c_{j, k} u_{\mu, j} \in C^\infty(G/\Gamma).$  Using formula (36) of \cite{FF}, we have that for any $s \geq \frac{\Re\nu}{2}$, 
\begin{align}\label{106}
\|\sum_{j \in A_{\mu}} c_{j, k} u_{\mu, j}\|_s^2 \nonumber & \approx \nonumber
\sum_{j \in A_{\mu}} |c_{j, k}^{\mu, \theta}|^2 \|u_{\mu, j} \|_s^2 \|u_{\theta, k}\|_s^2 (1 + |k|)^{-2s + \Re\nu^\theta} \nonumber
\\
& = C \sum_{j \in A_{\mu}} |c_{j, k}^{\mu, \theta}|^2 \|u_{\mu, j} \otimes u_{\theta, k}\|_{2s}^2 (1 + |k|)^{-2s + \Re\nu^\theta} \nonumber
\\
& 
\leq C \sum_{(j, k) \in A_{\mu}\times A_\theta} |c_{j, k}|^2 \|u_{\mu, j} \otimes u_{\theta, k}\|_{2s}^2 \nonumber
\\ 
& = C \| f_{\mu, \theta} \|_{2s}^2 < \infty. 
\end{align}

Then $\mathcal D\left(\sum_{j \in A_{\mu}} c_{j, k} u_{\mu, j} \right) = \sum_{j \in A_\mu} \mathcal D\left( c_{j, k} u_{\mu, j}\right)$ for all $k$. 
Then applying Lemma $\ref{2}$, we have that for all $\mathcal D \in \mathcal I(\mathcal H_\mu)$, 
\begin{align*}
\sum_{k \in A_\theta} \mathcal D\left(\sum_{j \in A_{\mu}} c_{j, k}^{\mu, \theta} u_{\mu, j}\right) u_{\theta, k}\nonumber & 
= \sum_{(j, k) \in A_\theta \times A_{\mu}} c_{j, k}^{\mu, \theta} \mathcal D (u_{\mu, j}) u_{\theta, k} 
\\
& = \tilde{\mathcal D}^1(f_{\mu, \theta})\\
&  = 0.
\end{align*}  
Now because $\{u_{\theta, k}\}_{k \in A_\theta}$ is a basis for $\mathcal H_\theta$, we conclude that for all $k$, \begin{equation}\label{D^1 vanish}\mathcal D\left(\sum_{j \in A_{\mu}} c_{j, k}^{\mu, \theta} u_{\mu, j}\right) = 0.\end{equation}



Recall that $\Box_1$ has a spectral gap at zero, so Theorem  1.2 of \cite{T} together with $(\ref{106})$ and $(\ref{D^1 vanish})$ show that there is some $P_{k}^{\mu, \theta} \in \mathcal{H}_\mu$ such that 
$$
L_1 P_{k}^{\mu, \theta} = \sum_{j \in A_{\mu}}c_{j, k}^{\mu, \theta} u_{\mu, j}
$$ 
and 
\begin{equation}\label{52}
\| P_{k}^{\mu, \theta} \|_r  \leq C_{r, \Gamma} \|\sum_{j \in A_{\mu}}c_{j, k}^{\mu, \theta} u_{\mu, j}\|_{3r + 4}.
\end{equation}
  
  Define 
  $$
  P_{\mu, \theta} := \sum_{k \in A_{\mu}} P_k^{\mu, \theta} \otimes u_{\theta, k}.
  $$  
  Observe that 
  \begin{align*}
  \| P_{\mu, \theta} \|_r & \leq \sum_{k \in A_{\theta}} \|P_k^{\mu, \theta} \otimes u_{\theta, k}\|_r
\\
& \leq \sum_{k \in A_\theta} \left( \|P_k\|_r \| u_{\theta, k}\| + \|P_k\| \|u_{\theta, k}\|_r\right)
\\
& \leq \sum_{k \in A_\theta} \sum_{j \in A_\mu} |c_{j, k}| \|u_{\mu, j}\|_{3r + 4} \| u_{\theta, k}\| + \sum_{k \in A_\theta} \sum_{j \in A_\mu} \|u_{\mu, j} \|_4 \|u_{\theta, k}\|_r
\\ 
&\leq C_{r, \Gamma} \left(\sum_{(j, k) \in A_\mu \times A_{\theta}} |c_{j, k}^{\mu, \theta}|^2 \|u_{\mu, j} \otimes u_{\theta, k}\|_{3r + 6}^2\right)^{1/2}
\\
& = C_{r, \Gamma} \| f_{\mu, \theta}\|_{3r + 6}.
\end{align*}
  
  Now define 
  $$
  P := \int_{\oplus (\mu, \theta) \in spec(\Box_1)\times spec(\Box_2)} P_{\mu, \theta}.
  $$  
  Then for every $r \geq 0$, 
  \begin{align}\label{glue} 
  \| P \|_{r}^2 & =  \int_{\oplus (\mu, \theta)} \|P_{\mu, \theta}\|_r^2\nonumber
\\
& \leq C_{r, \Gamma} \int_{\mu, \theta} \|f_{\mu, \theta}\|_{3 r + 6}^2 d\gamma(\mu, \theta) = C_{r, \Gamma} \| f \|_{3r + 6}^2 < \infty.
\end{align}

Using the relation 
$$
L_1L_2 = L_2 L_1,
$$
 we have 
 $$
 L_1 (L_2 P - g) = L_1 L_2 P  - L_1 g
 $$ 
 $$ 
 = L_2 L_1 P - L_1 g = L_2 f - L_1 g = 0.
 $$  
 Because $u_1$ is ergodic, it follows that $L_2 P - g = C$ on $G/\Gamma$, for some constant $C \in \mathbb{C}$.  Therefore, all irreducible components of $L_2 P - g$ are zero except possibly the trivial one.  Note $g$ has zero average, which implies $C = 0$, and we conclude the proof of Theorem  \ref{transfer}. $ \ \ \Box$

\section{Proof of Theorem \ref{split}}\label{s:split}

In this section we use the ideas from  \cite{DK} to construct the splitting.  However, the important difference is that the space of invariant distributions is infinite dimensional in each irreducible component for the horocycle map.  There is no \emph{a priori} reason why the same strategy as in \cite{DK} would work in the case of an infinite dimensional space of invariant distributions. However, in this particular situation we may use the fast decay of distributions at infinity, which is proven in \cite{T}. \\

Lemma \ref{inv_regularity} shows the invariant distributions for the time-one map are finite regularity.  At this point, we need the following more precise statement.  Section 3 of \cite{T} describes a countably infinite basis $\{\mathcal{D}_n\}_{n \in \Z}$ of invariant distributions for $\mathcal{I}(\mathcal{H})\cap W^{-r}(\mathcal H)$, where $\mathcal H$ is an irreducible unitary representation of $SL(2, \R)$ and $r > 1$.  Lemma 7.6 of \cite{T} shows there is some $C_{r, \mathcal H} > 0$ such that \begin{equation}\label{D_n-est}|\mathcal{D}_n(h)| \leq C_{r, \mathcal H} \| h\|_{3r + 8}n^{-(r + 2)}.\end{equation}

We prove Theorem \ref{split} in one irreducible component at a time, and then we are able to glue the estimates together using the existence of a spectral gap for irreducible lattices of $SL(2, \R) \times SL(2, \R)$.

\begin{lemma}
Consider an irreducible component $\mathcal{H}_\mu \otimes \mathcal{H}_{\theta}$ of $L^2(G/ \Gamma)$.  For $r \geq 0$, there exists a linear map $\mathcal{R}:W^{3 r + 10}(\mathcal{H}_\mu \otimes \mathcal{H}_{\theta})\rightarrow W^r(\mathcal{H}_\mu \otimes \mathcal{H}_{\theta})$ such that 
\begin{enumerate}
\item[(i)] $\|\mathcal{R} f\|_r \leq C_{s, \Gamma} \| f \|_{3r + 10}$.
\item[(ii)] $\tilde{D}_n^1(\mathcal{R} f) = \tilde{D}_n^1(f)$ for all $n \in \mathbb{Z}$.
\item[(iii)] $\mathcal{R}|_{Ann(\langle\{\tilde{\mathcal{D}}_n^1\}_{n \in \mathbb{Z}}\rangle)} = 0$.
\item[(iv)] $\mathcal{R}$ commutes with $L_2$.
\end{enumerate}
\end{lemma}
\begin{proof}

For each $n \in \Z$, let $\gamma_n \in C^{\infty}(\mathcal{H}_\mu)$ be dual to $\mathcal{D}_n$.  By Lemma B.1 in \cite{T}, we can label the distributions $D_n$ and choose the $\gamma_n$ so that $\|\gamma_n\|_r \leq C_r n^r.$

(i)  Define $\mathcal{R}$ on simple tensors $h_1\otimes h_2 \in C^\infty(\mathcal H_\mu \otimes \mathcal H_\theta)$ by $$\mathcal{R} (h_1\otimes h_2):= 
\sum_{n \in \Z} \gamma_n \otimes \tilde{D}_n^1(h_1\otimes h_2).$$   By $(\ref{D_n-est})$, it follows that \begin{equation}\label{simple tensors}\|\mathcal R (h_1\otimes h_2)\|_r \leq C_r \| h_1 \otimes h_2\|_{3r + 8},\end{equation} so that $\mathcal R$ is bounded on simple tensors.  Extend $\mathcal R$ linearly to finite sums of simple tensors in $\mathcal H$. 

Then let $f\in C^\infty(\mathcal H_\mu \otimes \mathcal H_\theta)$, and write its decomposition with respect to the basis $\{u_{\mu, j}\otimes u_{\theta, k}\}_{j, k}$, as in $(\ref{f decomp}).$  Define 
\begin{equation}
\label{R}
\mathcal R(f) := \sum_{k \in A_\theta} \mathcal R\left(\left( \sum_{j \in \mathcal A_\mu} c_{j, k} u_{\mu, j} \right) \otimes u_{\theta, k}\right).
\end{equation} 

Then by triangle inequality and (\ref{simple tensors}), we have 
\begin{align}\label{equa:triangle_twice}
\|\mathcal R f\|_{r} &= \|\sum_{k \in A_\theta} \mathcal R\left(\left( \sum_{j \in \mathcal A_\mu} c_{j, k} u_{\mu, j} \right) \otimes u_{\theta, k}\right)\|_{r} \nonumber
\\ 
& \leq C_{r, \Gamma} \sum_{k \in A_\theta} \| \left(\sum_{j \in \mathcal A_\mu} c_{j, k} u_{\mu, j}\right) \otimes u_{\theta, k}\|_{3 r + 8} \nonumber
\\ 
& \leq C_{r, \Gamma} \sum_{(j, k) \in A_\theta \times A_\mu} |c_{j, k}| \|u_{\mu, j} \otimes u_{\theta, k}\|_{3 r + 8}. 
\end{align}
Multiplying and dividing the summand by $(1 + j^2 + k^2)$, Cauchy-Schwarz gives 
\begin{align*}
(\ref{equa:triangle_twice}) 
& \leq C_{r, \Gamma} \left(\sum_{(j, k) \in A_\theta \times A_\mu} |c_{j, k}|^2 (1 + j^2 + k^2)^2\|u_{\mu, j} \otimes u_{\theta, k}\|_{3 r + 8}^2\right)^{1/2}
\\  
&  \leq C_{r, \Gamma} \| f \|_{3r + 10}.
\end{align*}

(ii)  
Again, let $f$ be as in $(\ref{f decomp}$).  By the proof of Part (i), the series defining $\mathcal R f$ converges unconditionally in $W^r(\mathcal H_\mu \otimes \mathcal H_\theta)$, and by Lemma $\ref{D^j regular}$, $\tilde{\mathcal D}_m^1$ is a bounded operator on this space.  Then we can interchange $\tilde{\mathcal{D}}_m^1$ with the sums and conclude 
\begin{align}\label{interchange} 
\tilde{\mathcal{D}}_m^1(\mathcal{R} f) & = \tilde{\mathcal{D}}_m^1\left(\sum_{k \in A_\theta} \sum_{n \in \Z} \gamma_n \otimes D_n\left( \sum_{j \in A_\mu} c_{j, k} u_{\mu, j}\right) u_{\theta, k}\right)\nonumber
\\
& = \sum_{k \in A_\theta} \sum_{n \in \Z} \tilde{\mathcal{D}}_m^1\left( \gamma_n \otimes D_n\left( \sum_{j \in A_\mu} c_{j, k} u_{\mu, j}\right) u_{\theta, k}\right).
\end{align}
Now because $\mathcal D_m$ is dual to $\gamma_m$, we get  
$$
(\ref{interchange}) = \sum_{k \in A_\theta} D_m\left( \sum_{j \in A_\mu} c_{j, k} u_{\mu, j}\right) u_{\theta, k} = \tilde{\mathcal D}_m^1 (f).
$$ 

(iii)  Let $f\in Ann(\langle \{\tilde{\mathcal D}_n^1\}_{n \in \Z}\rangle)$ be as in ($\ref{f decomp}$).  Again, we may interchange sums by Part (i) and by $(\ref{106})$ to get 
\begin{align*}
\mathcal R f & = \sum_{n \in \Z} \sum_{k \in A_\theta} \gamma_n \otimes D_n\left(\sum_{j \in A_\mu} c_{j, k} u_{\mu, j}\right) u_{\theta, k}\\
&  = \sum_{n \in \Z} \gamma_n \otimes \sum_{k \in A_\theta} D_n\left(\sum_{j \in A_\mu} c_{j, k} u_{\mu, j}\right) u_{\theta, k}\\
& = \sum_{n \in \Z} \gamma_n \otimes \sum_{(j, k)\in A_\mu \times A_\theta} c_{j, k} D_n(u_{\mu, j}) u_{\theta, k}\\
&  = \sum_{n \in \Z} \gamma_n \otimes \tilde{\mathcal D}_n^1(f) = 0.
\end{align*}

(iv)  The operator $L_2$ is bounded on $W^{3r + 9}(\mathcal{H}_\mu)$, so $$L_2 \mathcal{R}(f) = \left( \sum_{k \in A_\theta} \sum_{n \in \Z} D_n\left(\sum_{j \in A_\mu} c_{j, k} u_{\mu, j}\right) \gamma_n \otimes L_2 u_{\theta, k})\right)$$ $$ = \mathcal R\left( \sum_{(j, k) \in A_\mu \times A_\theta} c_{j, k} u_{\mu, j} \otimes L_2 u_{\theta, k})\right).$$

\end{proof}

$\textit{Proof of Theorem }\ref{split}:$  With this, the same argument used to prove Theorem 3.2 of \cite{DK} proves Theorem $\ref{split}$.  We provide it here for the convenience of the reader.  

Define $\mathcal R^\bot(f) = f - \mathcal R f.$  Then property (ii) implies $\tilde{\mathcal D}_n^1(\mathcal R^\bot f) = 0$ for every $n$.  Then property (i) and the argument below $(\ref{D^1 vanish})$ gives a solution $P \in C^\infty (\mathcal H)$ satisfying $L_1 P = \mathcal R^\bot f,$ and 
\begin{align*}
\| P \|_r &\leq C_{r, \Gamma} \| \mathcal R^\bot f\|_{3r + 6}
\\
& \leq C_{r, \Gamma} \| f\|_{3(3r + 6) + 10}  = C_r \| f \|_{9 r + 28}.
\end{align*}  

Moreover,  $$L_2 f - L_1 g = \phi,$$ so by property (iii) and (i v), 
\begin{align*} 
\mathcal R^\bot \phi & = \mathcal R^\bot (L_2 f - L_1 g) \\
& = \mathcal R^\bot L_2 f - L_1 g = L_2 \mathcal R^\bot f - L_1 g\\
& =  L_2 L_1 P  - L_1 g = L_1 (L_2 P - g).
\end{align*}

Because $L_1 (L_2 P - g) \in Ann(\{\tilde{\mathcal D}_n^1\}_{n \in \Z})$, we again use the argument below $(\ref{D^1 vanish})$ to conclude $$\| L_2 P - g\|_r \leq \| \mathcal R^\bot \phi\|_r \leq C_{r, \Gamma} \|\phi\|_{3r + 10}.$$  

The estimate of $\|L_1 P - f\|_r$ follows analogously.

Now we glue estimates from each irreducible component, as in equation $(\ref{glue})$.  This concludes the proof of Theorem \ref{split}.  
$\ \ \Box$



\section{Cohomology with coefficients in vector fields, splitting and conjecture on rigidity of perturbations}\label{s:splitVF}

The motivation for studying the first cohomology over $\mathbb Z^2$  parabolic actions on $SL(2, \mathbb R)\times SL(2, \mathbb R)/\Gamma$ is to understand the local structure about these actions. For the corresponding continuous time action, which is in fact  a  \emph{maximal} unipotent abelian action on $SL(2, \mathbb R)\times SL(2, \mathbb R)/\Gamma$, the local picture has been understood to large extent in \cite{DK}. Namely, the continuous time action is transversally locally rigid. A similar picture for the discrete time sub-action of the continuous time unipotent action would be a stronger result. 

A general approach towards understanding perturbations  of an action is to study the first cohomology over the the given $\mathbb Z^2$ action $\rho$  \emph{with coefficients in vector fields}, rather than only the cohomology with coefficients in smooth functions (see Subsection \ref{s:cohomology}). This is because the formal tangent space at $\rho$ to the space of all smooth $\mathbb Z^2$ actions on the given manifold $M$ is precisely the space of smooth cocycles over $\rho$ with coefficients in vector fields. Essentially, the first cohomology with coefficients in vector fields can be thought of as the first approximation to the set of actions in a small neighborhood of the given action, modulo smooth conjugacy classes. 
Now, if it turns out that the cohomology $H_\rho^1(\mathbb Z^2, Vect^\infty(M)) $ is finite dimensional  and well understood, and if there is a splitting in cohomology which is tame, then it may be possible to carry out an iterative procedure (similar to classical KAM iterative scheme, or proofs of generalized implicit function theorems) to obtain a result on perturbations of the action. For smooth Lie group actions a general result of this type is obtained in \cite{D}. So far a similar approach has not been applied to discrete abelian group actions with non-trivial but finite dimensional  $H_\rho^1(\mathbb Z^2, Vect^\infty(M))$ and the examples in this paper may be the first ones amenable to this technique. We remark further that for parabolic actions one does not have rich geometric structure as for partially hyperbolic actions, so to study general perturbations of such actions the only successful approach so far has been the one mentioned above: via the first cohomology and an iterative procedure.   

In this section we show that this approach is very likely to be successful in the context of parabolic $\mathbb Z^2$ actions on $SL(2, \mathbb R)\times SL(2, \mathbb R)/\Gamma$: we prove that the first cohomology with coefficients in vector fields is indeed finite dimensional and fairly simple to describe. After that, using the results from previous sections we show that there is a splitting in cohomology via an explicit construction. Finally, based on these results we state a conjecture on the local structure about these $\mathbb Z^2$ parabolic actions.

\subsection{Action induced operators on vector fields}
If $v$ is a vector in $\R^3$, let $v^\tau$ be its transpose.  Let \begin{equation}\label{coordinates}h(x) U + g(x) X + f(x) V := (h(x), g(x), f(x))^{\tau}.\end{equation}

\begin{lemma}\label{lower_star}
Let $\sigma(x) = (h(x), g(x), f(x))^\tau$ be given in (\ref{coordinates}).  Then 
\begin{align*}(e^{U})_*\sigma(x) & := D e^{-U} \cdot \sigma\circ e^{U} \\
& = \left(\begin{array}{rrr} 1 &  1 & -1 \\ 0 & 1 & -2 \\ 0 & 0 & 1\end{array}\right) (h(x), g(x), f(x))^\tau.
\end{align*}
\end{lemma}

Lemma \ref{lower_star} will follow immediately from Claim \ref{commutation}.


\begin{claim}\label{commutation}
We have 
$$
\begin{array}{lll}\frac{d}{dt}(e^{-U} e^{t U})|_{t = 0} = U e^{-U}\\  \frac{d}{dt}(e^{-U} e^{t X})|_{t = 0} = (X + U) e^{-U} \\ \frac{d}{dt}(e^{-U} e^{t V})|_{t = 0} = (V - 2 X - U) e^{-U}.\end{array}$$
\end{claim}
\begin{proof}

We have 
$$= 
 e^{te^{ad_{-U}}(X)} e^{-U}$$ Note $ad_{-U}(X) = [-U, X] = U$ so
$$\frac{d}{dt}(e^{U} e^{t X})|_{t = 0} = (X + U) e^{-U}.$$

Similarly, $$e^{-U} e^{t V} = e^{t(e^{ad_{-U}}(V))} e^{-U},$$ so   
$$
\frac{d}{dt}(e^{-U} e^{t V})|_{t = 0} = (V - 2 X - U) e^{-U}. 
$$
\end{proof}

This concludes the proof of Lemma \ref{lower_star}. $ \ \ \Box$ \\

The action induced first coboundary operators on vector fields in $sl(2, \R) \times sl(2, \R)$ are described in Section \ref{s:cohomology}, equation (\ref{cohom_vector}).  Let 
$$
\begin{array}{cc}
(\delta_v^1)((1, 0)) = \mathbb{L}_1 H \\
(\delta_v^1)((0, 1)) = \mathbb{L}_2 H,
\end{array}
$$
where we use $(1, 0)$ and $(0, 1)$ as generators of the acting group $\Z^2.$
Then $\mathbb L_1H = \left( e^{T U_1} \right)_*H-H$ and $\mathbb L_2H= \left(e^{S U_2}\right)_* H-H$, for any vector field $H\in \rm{Vect}^\infty(G / \Gamma)$.  
From (\ref{cohom_vector}), the cocycle generated by the vector field $(F, Y)$ is in Ker$(\delta_v^2)$ if and only if (\ref{equa:cocycle_eqn}) is satisfied.


\subsection{Trivialization of cohomology with coefficients in vector fields and splitting}\label{s:vect}

The following proposition describes the first coboundaries with coefficients in constant vector fields.

\begin{proposition}\label{constant_vf}
Let $F = (h_1, g_1, f_1, h_3, g_3, f_3)^\tau, Y= (h_2, g_2, f_2, h_4, g_4, f_4)^\tau \in sl(2, \R)$ $ \times sl(2, \R)$ be constant coefficient vector fields.  If  
\begin{equation}\label{equa:cocycle_eqn}
\mathbb L_2 F - \mathbb L_1 Y = 0,
\end{equation}
then 
\begin{equation}\label{equa:cocycle}
F \in \R U_1 + \R V_1 + \R X_1  + \R U_2  , \ \ Y \in \R U_2+ \R V_2 + \R X_2 + \R U_1.
\end{equation}

Furthermore, if $(F,  Y)$ is a coboundary, then 
\begin{equation}\label{equa:coboundary}
F \in \R U_1 + \R X_1, \ \ Y \in \R U_2 + \R X_2.
\end{equation}
Therefore the cohomology over the given $\mathbb Z^2$ parabolic action with coefficients in constant vector fields is 4-dimensional and the cohomology classes are parametrized by cocycles $(F, Y)$ where $F\in \R U_2+\R V_1$ and $Y\in \R U_1+ \R V_2$.
\end{proposition}

The proof of Proposition \ref{constant_vf} will be apparent from the proof of Proposition \ref{splitVF}, so we defer it until then.  Proposition \ref{splitVF} shows that the first cohomology with coefficients in smooth vector fields reduces to the first cohomology with constant vector fields.  Notice that by integrating each coefficient over $G / \Gamma$, we obtain a constant coefficient vector field.  For a smooth vector field $F$ we denote by $Ave F$ the constant vector field obtained from $F$ by taking averages of all coordinate functions of $F$ in the basis of $sl(2, \mathbb R)\times sl(2, \mathbb R)$ described above. We say $(Ave F, Ave Y)$ is in the trivial cohomology class if (\ref{equa:coboundary}) is satisfied for $(Ave F, Ave Y)$.
\begin{proposition}\label{splitVF}
If $F, Y, \Phi \in  \rm{Vect}^\infty(G / \Gamma)$ and $\mathbb L_2 F - \mathbb L_1 Y = \Phi$. Assume also that $(Ave F, Ave Y) $ 
is in the trivial cohomology class.  Then there exist $H, \tilde F, \tilde Y \in  \rm{Vect}^\infty(G / \Gamma)$ such that $F= \mathbb L_1H+\tilde F$ and $Y= \mathbb L_2H+\tilde Y$, and for every $r>0$ and $s\ge 27r+130$ there is constant $C_{s, r, S, T, \Gamma}  > 0$ such that the following estimates hold:
$$
\|\tilde F\|_r \leq C_{s, r, S, T, \Gamma} \| \Phi\|_s,
$$ 
$$
\|\tilde Y\|_r \leq C_{s, r, S, T, \Gamma}\| \Phi\|_s,
$$ and 
$$
\| H \|_r \leq C_{s, r, S, T, \Gamma}\| F\|_s.
$$
Moreover, for $(S, T)$ in an $\epsilon$ neighborhood of $(1,1)$, the constants   $C_{s, r, S, T, \Gamma}$ are uniformly bounded in $S, T$. 
\end{proposition}

\begin{proof}

We first prove the claim for $T=S=1$. Recall that $u_1 := e^{U_1}$ and $u_2 := e^{U_2}.$  Using the coordinates (\ref{coordinates}), let $F = (h_1, g_1, f_1, h_3, g_3, f_3)^\tau $,  $Y = (h_2, g_2, f_2, h_4, g_4, f_4)^\tau $,  and $\Phi = (\phi^1, \phi^2, \phi^3, \phi^4, \phi^5, \phi^6)^\tau$.  

$\mathbb L_2 F - \mathbb L_1 Y = \Phi$ implies 
$$
\left(\begin{array}{cccccc} h_1 \\ g_1 \\ f_1 \\ h_3  -2 g_3  - f_3 \\  g_3 +  f_3  \\  f_3  \end{array}\right)\circ u_2  - \left(\begin{array}{llllll} h_1 \\ g_1 \\ f_1 \\ h_3 \\ g_3 \\ f_3\end{array}\right)
$$ 
\begin{equation}\label{cocycle_matrix} 
- \left(\begin{array}{cccccc} h_2  -2 g_2  - f_2 \\  g_2 +  f_2  \\  f_2 \\  h_4 \\  g_4 \\  f_4 \end{array}\right)\circ u_1 + \left(\begin{array}{lll} h_2  \\ g_2 \\ f_2 \\  h_4 \\ g_4 \\ f_4\end{array}\right) = \left(\begin{array}{llllll} \phi^1\\ \phi^2 \\ \phi^3 \\ \phi^4 \\ \phi^5 \\ \phi^6 \end{array}\right).
\end{equation}

We consider the first three coordinates.  Then $L_2 f_1 - L_1 f_2 = \phi_3.$  Then Theorem \ref{split} implies that if $Ave f_1=Ave f_2=0$ then  there is some $P \in C^\infty(G / \Gamma)$ and a constant $C_{r, \Gamma} > 0$ such that 
\begin{equation}\label{f_estimate}
\begin{array}{ll} 
\|f_1 - L_1P\|_r \leq C_{r, \Gamma} \|\phi_3\|_{3r + 10} \\
\|f_2 - L_2P\|_r \leq C_{r, \Gamma} \|\phi_3\|_{3r + 10}.
\end{array}
\end{equation}
Next, 
$$
L_2g_1 - L_1g_2 - f_2 \circ u_1 = \phi_2.
$$  
By (\ref{f_estimate}), there is some $\tilde \phi_3 \in C^\infty(G / \Gamma)$ such that 
$$
f_2 = L_2 P + \tilde \phi_3 ,
$$
where $\|\tilde \phi_3\|_r \leq C_{r, \Gamma} \|\phi_3 \|_{3r + 10}$ for all $r \geq 0$.  Therefore, 
$$
L_2(g_1 - P \circ u_1) - L_1 g_2 = \phi_2 + \tilde \phi_3 \circ u_1.
$$
Applying Theorem \ref{split} again shows that if $Ave g_2=0$ then there is some $Q \in C^\infty(G / \Gamma)$ such that 
\begin{equation}\label{g_estimate}
\begin{array}{ll}
\|(g_1 - P \circ u_1) - L_1 Q\|_r \leq C_{r, \Gamma} \|\phi_2 + \tilde \phi_3\|_{3r + 10} \\
\|g_2 - L_2 Q\|_r \leq C_{r, \Gamma} \|\phi_2 + \tilde \phi_3\|_{3r + 10}.
\end{array}
\end{equation}
Notice that we did not need to assume that $Ave g_1=0$ because we have a freedom to choose $P$ up to a constant, so we chose it so that $Ave( g_1-P\circ u_1)=0$.
Now 
$$
L_2 h_1 - L_1 h_2 + 2 g_2 \circ u_1 + f_2 \circ u_1 = \phi_1.
$$
By (\ref{g_estimate}), 
$$
g_2  = L_2 Q + \tilde \phi_2,
$$
where $\|\tilde \phi_2\|_r \leq C_{r, \Gamma} \|\phi_2 - \tilde \phi_3\|_{3r + 10}$ for all $r \geq 0$.
Then 
$$
L_2(h_1 + 2 Q \circ u_1 + P \circ u_1) - L_1 h_2 = \phi_1 - 2 \tilde \phi_2 \circ u_1 - \tilde \phi_3 \circ u_1 .
$$
By Theorem \ref{split} again, if $Ave h_2=0$ (and after adjusting $Q$ by a constant so that $Ave(h_1 + 2 Q \circ u_1 + P \circ u_1)=0$), there is some $R \in C^\infty(G / \Gamma)$ such that 
$$
\begin{array}{ll}
\|(h_1 + 2 Q \circ u_1 + P\circ u_1 - L_1 R\|_r \leq C_{r, \Gamma} \|\phi_1-  2\tilde \phi_2 \circ u_1 - \tilde \phi_3 \circ u_1\|_{3r + 10} \\
\|h_2 - L_2 R\|_r \leq C_{r, \Gamma} \|\phi_1-  2\tilde \phi_2 \circ u_2 - \tilde \phi_1 \circ u_2\|_{3r + 10}.
\end{array}
$$

Then repeating the process for the bottom three coordinates, we find a vector field $H$ and a remainder $\tilde F \in \rm{Vect}^\infty(G / \Gamma)$ such that  
$$
F = \mathbb L_1 H + \tilde F,
$$
where 
$$
\mathbb L_1 = \left(\begin{array}{cccccc} L_1 & -2 (u_1)^* & -(u_1)^* \\  & L_1 & (u_1)^* \\  &  & L_1\\  & & & L_1 &  & \\ & & &   & L_1 &  \\ & & &  &  & L_1 \end{array}\right)
$$
and $(u_1)^* f = f \circ u_1$.  
Similarly, $Y = \mathbb L_2H+\tilde Y$.  The following estimates are satisfied:
\begin{equation*}
\begin{aligned}
\|\tilde F\|_r&\le C_{s,r, \Gamma}\|\Phi\|_s\\ 
\|\tilde Y\|_r&\le C_{s,r, \Gamma}\|\Phi\|_s\\ 
\|H\|_r&\le C_{s,r, \Gamma}\|F\|_s\\ 
\end{aligned}
\end{equation*}
where $s\ge 3 ( 3 (3r + 10) + 10) + 10 = 27r+130$.

Lastly and in a completely analogous manor, the same holds for the maps $e^{T U_1}$ and $e^{S U_2}$.  The fact that the constant $C_{s, r, S, T, \Gamma} > 0$ is uniformly bounded in $(S, T)$ in an $\epsilon > 0$ neighborhood of (1, 1) follows from the estimates in the proof of Theorem 1.2 of \cite{T}.
\end{proof}

\textit{Proof of Proposition }\ref{constant_vf}:

We first consider the case when $(F, Y)$ is a (constant coefficient) cocycle.  Now $\Phi = 0$, so (\ref{cocycle_matrix}) implies (\ref{equa:cocycle}) after all the coefficients are substituted by constants, namely that in this case$f_2=f_3=g_2=g_3=0$.  When $(F, Y)$ reduces to a coboundary then the proof above shows that the necessary conditions for that are: $f_1=h_3=g_3=f_3=h_2=g_2=f_2=f_4=0$.  This implies (\ref{equa:coboundary}) and concludes the proof of Proposition \ref{constant_vf}. $\ \ \Box$

\subsection{Transversal local rigidity conjecture}

Consider the family of $\mathbb Z^2$ actions $\{\rho_{\lambda}\}$ on $SL(2, \mathbb R)\times SL(2, \mathbb R)/\Gamma$, where $\lambda=(\lambda_1, \lambda_2)\in B_\epsilon(0,0)\subset \mathbb R^2$ and each $\rho_{\lambda}$ is homogeneous action generated by the following elements of $SL(2, \mathbb R)\times SL(2, \mathbb R)$: $$
\left(\begin{array}{cc}
 1& S\\
  \lambda_1& 1  
\end{array}
\right)\times \rm{Id},\,\,\,\,
\rm{Id}\times \left(\begin{array}{cc}
 1& T\\
  \lambda_2& 1  
\end{array}
\right)$$ 

\begin{conj} Every sufficiently transversal and sufficiently small perturbation $\{\tilde \rho_\lambda\}$ of the family  $\{ \rho_\lambda\}$ in a neighborhood of $\lambda=(0,0)$, contains a parameter $\bar\lambda$ such that  $\tilde \rho_{\bar \lambda}$ is a parabolic action smoothly conjugate to the homogeneous  $\mathbb Z^2$ action generated by 
$$
\left(\begin{array}{cc}
 1& T'\\
  0& 1  
\end{array}
\right)\times \rm{Id},\,\,\,\,
\rm{Id}\times \left(\begin{array}{cc}
 1& S'\\
  0& 1  
\end{array}
\right)$$ 
for some $T'$ and $S'$ in $\mathbb R$.

\end{conj}

We remark that if the estimates in the splitting construction in Proposition \ref{splitVF} were \emph{tame} i.e. if the loss of regularity for $\tilde F, \tilde Y$ and $H$ was \emph{fixed} and independent of $r$, for every given regularity $r$, then the conjecture would follow by the similar approach as in \cite{DK}, adapted to the discrete case. Now the fact that the estimates in Proposition \ref{splitVF} are not tame comes directly from not having tame estimates for the solution of the coboundary equation for functions over the discrete unipotent map on $SL(2, \mathbb R)/\Gamma$, obtained in \cite{T}. 

  The authors believe that obtaining tame estimates for the solution of the coboundary equation for functions over the discrete unipotent map on $SL(2, \mathbb R)/\Gamma$ is a  challenging problem.  The second author obtained the non tame estimates through an analysis in irreducible models \cite{T}.  It is not clear to the authors how to substantially improve these estimates, however a more precise analysis may be possible. 
 Flaminio-Forni proved tame estimates for the cohomological equation of the horocycle flow in \cite{FF}, but their method does not seem to work for the horocycle map. The important point for the argument of Flaminio-Forni is that the restriction of the generator for the horocycle flow to each irreducible component $\mathcal H_\mu$ of $L^2(SL(2, \R) / \Gamma)$ takes a relatively simple form in terms of the bases $\{u_{\mu, k}\}_{k}$.  This allowed them to determine the invariant distributions by their values on basis elements in $\mathcal H_\mu$, and then construct a Green's operator.   The case of the time-one map is different.  The operator $e^U$ is a complicated infinite dimensional matrix in each basis $\{u_{\mu, k}\}_{k}$, and it becomes unclear how to proceed.  
 Another approach is to obtain tame estimates for cohomological equations that are more similar to that of the horocycle flow.  For example, it is possible to derive the cohomological equation for the time-1 horocycle map from the cohomological equations $(U + in) f = g$, for $n \in \Z$.

\end{document}